\documentclass[a4paper,english]{article}
\usepackage[T1]{fontenc}% encodage 8-bits + lettre accentuée en vectorielle;
\usepackage{textcomp}% améliore certains symboles de bases
\usepackage{lmodern}% remplace la police ComputerModern par LatinModern (+ mieux bien)
\usepackage[utf8]{inputenc}% lettres accentuées tapée directement
\usepackage{babel}% pour les césures automatiques, etc.
\usepackage{mathtools}% compléments à amsmath
\usepackage{amssymb}%symboles mathématiques supplémentaires, notamment pour la flèche double
\usepackage{amsthm}% environnement ams-theorem
\newtheorem{lemma}{Lemma}[section]

\newtheorem{theorem}[lemma]{Theorem}
\newtheorem{theorema}{Theorem}

\theoremstyle{definition}

\newtheorem{example}[lemma]{Example}
\DeclareMathOperator{\Aut}{Aut}% Cayley
\DeclareMathOperator{\cayl}{Cay}% Cayley
\DeclareMathOperator{\col}{col}% Cayley
\DeclareMathOperator{\id}{id}
\DeclareMathOperator{\lab}{lab}

\DeclareMathOperator{\Stab}{Stab}% square function
\DeclareMathOperator{\Sym}{Sym}% square function
\DeclarePairedDelimiter{\abs}{\lvert}{\rvert}%  valeurs absolues
\newcommand*{\presentation}[2]{\langle#1\mid#2\rangle}% presentation
\newcommand*{\setst}[2]{\{#1 \mid #2\}}%  set 1, such that 2
\newcommand*{\unCayley}[2]{\cayl(#1,#2)}% Cayley
\newcommand*{\colCayley}[2]{\cayl_{\col}(#1,#2)}% Color Cayley
\newcommand*{\labCayley}[2]{\cayl_{\lab}(#1,#2)}% Color Cayley
\newcommand*{\Z}{\mathbf{Z}}% Z
\usepackage{hyperref} % PDF hyperliens : coloriés; sur plusieurs lignes; signet; vrai numéro (pas arabe)
\title{Most rigid representation and Cayley index of finitely generated groups}
\author{Paul-Henry Leemann\thanks{Supported by grant 200021\textunderscore188578 of the Swiss National Fund for Scientific Research.}, Mikael de la Salle\thanks{Supported by ANR grants AGIRA ANR-16-CE40-0022 and Noncommutative analysis on groups and quantum groups ANR-19-CE40-0002-01.}}
\date{\today}
\hypersetup{pdftitle={Most rigid representation and Cayley index of finitely generated groups}, pdfauthor={Paul-Henry Leemann, Mikael de la Salle}, pdfsubject={GRR and MRR for infinite groups.},pdfkeywords={Cayley index, GRR, MRR, Cayley graph, automorphisms of graphs, generalized dihedral group,  regular automorphism group}}
\begin{document}
\maketitle
%
%
%
%
%
%
%
%
%
%
%%%%%%%%%%%%%%%%%%%%%%%%%%%%%%%%%%%%%%%% Abstract %%%%%%%%%%%%%%%%%%%%%%%%%%%%%%%%%%%%%%%%
\begin{abstract}
  If $G$ is a group and $S$ a generating set, $G$ canonically embeds into the automorphism group of its Cayley graph and it is natural to try to minimize, over all generating sets, the index of this inclusion. This infimum is called the Cayley index of the group. In a recent series of works, we have characterized the infinite finitely generated groups with Cayley index $1$. We complement this characterization by showing that the Cayley index is $2$ in the remaining cases and is attained for a finite generating set.
\end{abstract}
\section{Introduction}
Given a group $G$ and a symmetric generating set $S \subset G\setminus\{1\}$, the Cayley graph $\unCayley{G}{S}$ is the simple unoriented graph with vertex set $G$ and an edge between $g$ and $h$ precisely when $g^{-1}h \in S$.
By construction, the action by left-multiplication of $G$ on itself induces an action of the group on its Cayley graph, which is free and vertex-transitive.
One can then define the \emph{Cayley index} of $G$ as the infimum over all symmetric generating sets of $[\Aut(\unCayley{G}{S}):G]$.
Cayley graphs of $G$ with $[\Aut(\unCayley{G}{S}):G]$ minimal are known under the name of \emph{most rigid representations}; if the index is $1$ they are called \emph{graphical rigid representation} (usually written GRR).

Let us now quickly recall that a generalized dicyclic group is a non-abelian group with an abelian subgroup $A$ of index $2$ and an element $x$ not in $A$ such that $x^4=1$ and $xax^{-1} =a^{-1}$ for every $a\in A$.

Thanks to combined efforts of, notably, Imrich, Watkins, Nowitz, Hetzel and Godsil, see \cite{MR642043} and the references therein, a complete classification of finite groups admitting a GRR was obtained in the 70's. More precisely, there is list of $13$ exceptional groups (of order at most $32$) such that a finite group has Cayley index $1$ if and only if it is neither an abelian group with an element of order at least $3$, nor a generalized dicyclic group, nor one of the $13$ exceptional groups.
The computation of the Cayley index of the remaining finite groups has been initiated Imrich and Watkins \cite{MR0457275} and recently completed by Morris and Tymburski \cite{MR3995536}.

On the other hand, until recently very little was known for infinite finitely generated groups. In \cite{LdlS2020,LdlS2021}, confirming a conjecture of Watkins, we showed that an infinite finitely generated group has Cayley index $1$ if and only if it is neither an abelian group, nor a generalized dicyclic group; in which case the Cayley index is attained for a finite generating set. This completed the classification of finitely generated groups admitting a GRR. The aim of this to compute the Cayley index of all finitely generated groups.
Our main result, which completes the classification given in \cite{MR3995536} for finite groups, is
\begin{theorema}\label{Thm:Main}
Let $G$ be an infinite finitely generated group.
Then its Cayley index is equal to $2$ if $G$ is abelian or generalized dicyclic and $1$ otherwise.
Moreover, this index is attained for some finite generating set $S$.
\end{theorema}

The difficult cases of the Theorem~\ref{Thm:Main} are contained in \cite{LdlS2020,LdlS2021}. The main contribution of this short note is to study separately the abelian and generalized dicyclic groups; this study also relies on the results from \cite{LdlS2020}.

In order to prove Theorem~\ref{Thm:Main}, we will consider the following three subgroups of automorphisms of $\unCayley{G}{S}$:
\begin{itemize}
\item
The full group $\Aut(\unCayley{G}{S})$ which consists of permutations $\varphi$ of $G$ satisfying $\varphi(g S) = \varphi(g)S$ for every $g \in G$.
 \item The group $\Aut(\colCayley{G}{S})$ of \emph{colour-preserving automorphisms}, defined as the permutations $\varphi$ satisfying $\varphi(g s) \in \{\varphi(g)s,\varphi(g) s^{-1}\}$ for every $g \in G$ and $s \in S^\pm$.
\item The group $\Aut(\labCayley{G}{S})$ of \emph{labelled-preserving automorphisms}, defined as the permutations of $G$ satisfying $\varphi(g s) = \varphi(g) s$ for every $g \in G$, $s \in S^\pm$. This group is isomorphic to $G$ acting by left-translation because $S$ is generating.
\end{itemize}
In \cite{MR3027684}, Byrne, Donner and Sibley computed the group $\Xi_G\coloneqq\Aut(\colCayley{G}{G})$ for $G$ finite.
We will generalize their result to all groups, not necessarily finitely generated.
Before stating it, we introduce some notation.
On the one hand, if $G$ is abelian we denote by $\eta$ the inverse map $g\mapsto g^{-1}$ on $G$. Since $G$ is abelian, $G\rtimes\{\id,\eta\}$ is a subgroup of $\Aut(\colCayley{G}{S})$ for any $S$.
On the other hand, if $G=\langle A,x\rangle$ is a generalized dicyclic group, we denote by $\psi$ the permutation that is the identity on $A$ and the inverse on $xA$.
Once again, a simple verification shows that $G\rtimes\{\id,\psi\}$ is a subgroup of $\Aut(\colCayley{G}{S})$ for any $S$.
Finally, let $Q_8=\{\pm1,\pm i,\pm j,\pm k\}$ be the quaternion group.
For any choice of $(\varepsilon_i,\varepsilon_j,\varepsilon_k)\in\{-1,1\}^3$, we define a permutation of $\varphi_{(\varepsilon_i,\varepsilon_j,\varepsilon_k)}$ by $\varphi_{(\varepsilon_i,\varepsilon_j,\varepsilon_k)}(\pm x)=\pm x^{\varepsilon_x}$ for $x\in \{i,j,k\}$ and $\varphi_{(\varepsilon_i,\varepsilon_j,\varepsilon_k)}(\pm1)=\pm1$.
One easily checks that for any generating set $S$ of $Q_8$,
\[\Stab_{\Aut(\colCayley{Q_8}{S})}(1_{Q_8})=\setst{\varphi_{(\varepsilon_i,\varepsilon_j,\varepsilon_k)}}{(\varepsilon_i,\varepsilon_j,\varepsilon_k)\in\{-1,1\}^3}\cong(\Z/2\Z)^3\] and that $\Aut(\colCayley{Q_8}{S})=Q_8\cdot \Stab_{\Aut(\colCayley{Q_8}{S})}(1_{Q_8})\cong \Z/2\Z\wr(\Z/2\Z)^2$, see~\cite{MR3027684} for more details.
With these notations, we have the following generalization of the classification theorem of \cite{MR3027684}.
\begin{theorema}\label{Thm:Main2}
Let $G$ be a group and let $\Xi_G=\Aut(\colCayley{G}{G})$.
\begin{enumerate}
\item If $G=(\Z/2\Z)^{(I)}$, then $\Xi_G=G$,
\item If $G$ is any other abelian group, then $\Xi_G=G\rtimes\{\id,\eta\}$,
\item If $G=Q_8\times(\Z/2\Z)^{(I)}$, then $\Xi_G=\Xi_{Q_8}\times(\Z/2\Z)^{(I)}$,
\item If $G$ is any other generalized dicyclic group then $\Xi_G=G\rtimes\{\id,\psi\}$,
\item If $G$ is neither abelian nor generalized dicyclic, then $\Xi_G=G$.
\end{enumerate}
\end{theorema}
\paragraph{Acknowledgements} The question of determining the Cayley index of abelian and generalized dicyclic groups was asked to us by Adrien Le Boudec. We thank him for this question and for interesting discussions.
\section{The proofs}
For $G$ a group, recall that we denote by $\Xi_G$ the group $\Aut(\colCayley{G}{G})$. We additionally denote by $\xi_G$ the stabilizer of $1_G$ in $\Xi_G$.
Since $G$ acts transitively on the vertices of $\unCayley{G}{G}$ we have that $\Xi_G=G\cdot \xi_G$ and since the action of $G$ is free, $G\cap \xi_G=\{\id\}$.

Recall that a group $G$ is \emph{Boolean} if all its elements have order at most $2$. It is a well-known fact that this equivalent to $G$ being isomorphic to $(\Z/2\Z)^{(I)}$ for some $I$ (possibly empty).

We begin by generalizing \cite[Theorem 4]{MR3027684}.
\begin{lemma}\label{Lemma:Boolean}
Let $G$ be any group and $B$ be a Boolean group.
Then $\xi_{G\times B}\cong \xi_G$ and $\Xi_{G\times B}\cong \Xi_G\times B$.
\end{lemma}
\begin{proof}
  A straigthforward verification shows that
  \[ \xi_{G \times B} = \{ (g,b) \mapsto (\varphi(g),b) \mid \varphi \in \xi_G\}.\]
Indeed, the inclusion $\subset$ is clear, and the converse is because any $\psi \in \xi_{G \times B}$ preserves $G$ and satisfies $\psi(gh) = \psi(g) h$ whenever $h$ has order $2$.
\end{proof}
\begin{lemma}\label{lem:gendicyclic_not_Q8} Let $G=\langle A,x\rangle$ be a generalized dicyclic group that is not isomorphic to $Q_8 \times B$. There is $a \in A$ such that $a^2 \notin\{1,x^2\}$.
\end{lemma} 
\begin{proof} Assume that $a^2 \in \{1,x^2\}$ for every $a \in A$. In particular $A$ is an abelian group of exponent at most $4$ and hence by a theorem of Prüfer a direct sum of cyclic groups, that is in our case a direct sum of copies of $\Z/4\Z$ and of $\Z/2\Z$.
That is, $A=(\Z/2\Z)^{(I)}\times(\Z/4\Z)^{(J)}$ for some (possibly empty) sets $I$ and $J$. Firstly, $J$ is nonempty as otherwise $G$ would be abelian and hence not generalized dicyclic. Secondly, the assumption that $a^2 \in \{1,x^2\}$ for every $a \in A$ implies that $J$ has cardinality exactly one, and that the group generated by $\Z/4\Z^{(J)}$ and $x$ is isomorphic to $Q_8$. Therefore, $G$ is isomorphic to $Q_8 \times (\Z/2\Z)^{(I)}$. This proves the lemma.
  \end{proof}
We now proceed to prove Theorem~\ref{Thm:Main2}.
\begin{proof}[Proof of Theorem~\ref{Thm:Main2}]
By \cite[Theorem 7]{LdlS2020}, $\abs{\xi_G}=1$ if and only if $G$ is neither generalized dicyclic, nor abelian with an element of order at least $3$. This takes care of the first and the last assertion.

Now, let $G$ be an abelian group with an element $g_0$ of order at least $3$.
We will prove that $\xi_G = \{id,\eta\}$ where $\eta: g \mapsto g^{-1}$. Let $\varphi$ be any element of $\xi_G$. Replacing $\varphi$ by $\eta \circ \varphi$, we can suppose that $\varphi(g_0)=g_0$. If $h \in G$ is such that $\varphi(g_0h)\neq g_0h$, then
\begin{align*}
\varphi(g_0h)&=\varphi(1\cdot g_0h)=(g_0h)^{-1}=g_0^{-1}h^{-1}\\
			&=\varphi(g_0h)=g_0h^{-1}.
\end{align*}
This implies that $g_0^2=1$, a contradiction. 
Hence $\varphi(g_0h)=g_0h$ for every $h\in G$, that is $\varphi=\id$. This proves that $\xi_G=\{\id,\eta\}$.
The result for $\Xi_G$ follows.

The third assertion is Lemma~\ref{Lemma:Boolean}.

Finally, suppose that $G=\langle A,x\rangle$ is generalized dicyclic, but not isomorphic to $Q_8\times B$ with $B$ Boolean.
We will prove that $\xi_G=\{\id,\psi\}$ where $\psi(a)=a$ and $\psi(ax)=(ax)^{-1}$ for $a\in A$.
By Lemma~\ref{lem:gendicyclic_not_Q8} there exists an element $a_0\in A$ such that $a_0^2\notin\{1,x^2\}$. Let $\varphi$ be an element of $\xi_G$.
Since $\varphi(x)$ is in $\{x,x^{-1}\}$, composing by $\psi$ if necessary, we can suppose that $\varphi(x)=x$.
Moreover, since $\varphi_{|_A}$ belongs to $\xi_A$, we know from the second item that $\varphi_{|_A}$ is in $\{\id,(\cdot)^{-1}\}$.
Suppose that $\varphi_{|_A}=(\cdot)^{-1}$.
But then 
\begin{align*}
a_0&=\varphi(a_0^{-1})=\varphi(xa_0x^{-1})\in\{\varphi(x)a_0x^{-1},\varphi(x)xa_0^{-1}\}\\
&\in\{xa_0x^{-1}=a_0^{-1},x^2a_0^{-1}\}.
\end{align*}
Hence either $a_0=a_0^{-1}$ or $a_0^2=x^2$, which is absurd.
We have proved that for all $\varphi\in\xi_G$ such that $\varphi(x)=x$ we have $\varphi_{|_A}=\id$.
Now we look at $\tilde\varphi(g)\coloneqq x^{-1}\varphi(xg)$.
Since $\varphi(x)=x$ and $x^2$ is in $A$, the map $\tilde\varphi$ is in $\xi_G$ and $\tilde\varphi(x)=x^{-1}\varphi(x^2)=x$.
Hence $\tilde\varphi(a)=a$ for every $a$ in $A$, which implies that $\varphi(xa)=xa$ for every $a\in A$, that is $\varphi=\id$.
\end{proof}
If $G$ is a group and $S$ a generating set, we denote by $\Xi_S=\Xi_{G,S}$
the group $\Aut(\colCayley{G}{S})$ and by $\xi_S=\xi_{G,S}$ the stabilizer of $1_G$ in $\Aut(\colCayley{G}{S})$.
Once again we have $\Xi_S=G\cdot \xi_S$ with $G\cap \xi_S=\{\id\}$ and the group $\xi_S$ admits the following characterization:
\[
	\xi_S=\setst{\varphi\in\Sym(G)}{\varphi(1)=1\textnormal{ and }\varphi(gs)\in\varphi(g)\{s,s^{-1}\}\forall g\in G,\forall s\in S}.
\]
It follows that if $S\subset T$ are two generating sets of $G$, then $\xi_T\leq \xi_S$. Our next lemma uses a compactness argument to show that for a finitely generated $G$, the value $\inf\setst{[\Aut(\colCayley{G}{S}:G]}{S\textnormal{ generates } G}=\abs{\xi_G}$ is attained on a finite $S$. In section~\ref{sec:quantitative_results}, we will prove this directly, in order to obtain explicit bounds on $S$.
\begin{lemma}\label{Lemma:Compacity}
Let $G$ be a finitely generated group. Then there exists a finite generating set $S$ with $\xi_S=\xi_G$.
\end{lemma}
\begin{proof}
By Theorem~\ref{Thm:Main2} we know that $\xi_G$ has cardinality $1$, $2$ or $8$. In particular, it is finite.
Therefore there exists a finite subset $F\subset G$ such that for all $\varphi\in\xi_G$ if $\varphi_{|_F}=\id$ then $\varphi=\id$.

Let $S_0$ be a finite generating set of $G$. We claim that there exists $S_0\subset T$ finite such that for all $\varphi\in\xi_T$ if $\varphi_{|_F}=\id$ then $\varphi_{|_{S_0F}}=\id$.
Indeed, if it wasn't the case, we would have a $\varphi$ in $\bigcap_{S\textnormal{ finite}}\xi_S=\xi_G$ such that $\varphi_{|_F}=\id$ but $\varphi_{|_{S_0F}}\neq\id$, which is absurd.

Let $\varphi\in\xi_T$ such that $\varphi_{|_F}=\id$. Let $s$ be an element of $S_0$. Define $\varphi_1(g)\coloneqq s^{-1}\varphi(sg)$.
Then $\varphi_1$ belongs to $\xi_T$ and ${\varphi_1}_{|_F}=\id$, which implies that ${\varphi_1}_{|_{S_0F}}=\id$ and hence $\varphi_{|_{sS_0F}}=\id$.
Since this is true for all $s\in S_0$ we obtain that $\varphi_{|_{S_0^2F}}=\id$.
By induction $\varphi_{|_{S_0^nF}}=\id$ for all $n$ and hence $\varphi=\id$.
We have just proved that for all $\varphi\in\xi_T$ if $\varphi_{|_F}=\id$ then $\varphi=\id$.
Since $F$ is finite, this means that $\xi_T$ is finite.

To finish the proof, observe that $(\xi_{T^n})_n$ is a decreasing sequence of finite subgroups of $\xi_{T}$ whose intersection is equal to $\xi_G$. In particular, for $n$ big enough we have $\xi_{T^n}=\xi_G$.
\end{proof}
We finally proceed to prove Theorem~\ref{Thm:Main}:
\begin{proof}[Proof of  Theorem~\ref{Thm:Main}]
On the one hand, by Theorem~\ref{Thm:Main2} and Lemma~\ref{Lemma:Compacity}, there exists a finite generating set $S$ such that $[\Aut(\colCayley{G}{S}:G]$ is equal to $2$ if $G$ is abelian or generalized dicyclic and $1$ otherwise.
On the other hand, by \cite[Theorem 10]{LdlS2020} and \cite[Proposition 2.2]{LdlS2021}, there exists a finite $T$ containing $S$ such that $[\Aut(\unCayley{G}{T}:\Aut(\colCayley{G}{T}]=1$.
\end{proof}
\section{Quantitative bounds}\label{sec:quantitative_results}

The aim of this section is to prove a quantitative form of Lemma~\ref{Lemma:Compacity}, which remains true for non-finitely generated groups.
The argument is an adaptation of the proof of Theorem~\ref{Thm:Main2}. 
\begin{lemma}\label{Lemma:Trivial}
Let $S\subset T$ be two symmetric generating sets of $G$ and let $S_0$ be a subset of~$G$.
Suppose that for all $\varphi$ in $\xi_{T}$, if $\varphi_{|_{S_0}}=\id$ then $\varphi_{|_{S\cup SS_0}}=\id$.
Then for all $\varphi$ in $\xi_{T}$, if $\varphi_{|_{S_0}}=\id$ then $\varphi=\id$.
\end{lemma}
\begin{proof}
Let $\varphi$ be in $\xi_{T}$ such that $\varphi_{|_{S_0}}=\id$ and let $s$ be an element of $S$.
We define $\tilde\varphi_s(g)\coloneqq s^{-1}\varphi(sg)$. One verifies that $\tilde\varphi_s(g)$ is indeed in $\xi_{T}$ and 
$\tilde\varphi_{|_{S_0}}=\id$.
For $t\in S \cup S S_0$ we have $\varphi(st)=s\tilde\varphi_s(t)=s\cdot t$, that is $\varphi$ is the identity on $S^{\leq2} \cup S^{\leq 2} S_0$. By induction we obtain that $\varphi$ is the identity.
\end{proof}
\begin{theorem}\label{Thm:Quantitative}
Let $G$ be a group and $S$ a symmetric generating set.
\begin{enumerate}
\item If $G$ is Boolean, then $\Xi_S=\Xi_G=G$,
\item If $G$ is any other abelian group, then $\Xi_{S^{\leq 2}}=\Xi_G=G\rtimes\{\id,\eta\}$,
\item If $G=Q_8\times B$ with $B$ Boolean, then $\Xi_{S^{\leq 3}}=\Xi_G=\Xi_{Q_8}\times B$,
\item If $G$ is any other generalized dicyclic group then $\Xi_{S^{\leq 3}}=\Xi_G=G\rtimes\{\id,\psi\}$,
\item If $G$ is neither abelian nor generalized dicyclic, then $\Xi_{S^{\leq 3}}=\Xi_G=G$.
\end{enumerate}
\end{theorem}
\begin{proof}
The first assertion is obvious.

Now, let $G$ be an abelian group with an element $g_0$ of order at least $3$. We may choose such an element in $S$.
We will prove that $\xi_G = \{id,\eta\}$ where $\eta: g \mapsto g^{-1}$. Let $\varphi$ be any element of $\xi_{T}$ for $T = S^{\leq 2}$.
Suppose that $\varphi(g_0)=g_0$.
If $s \in T$ is such that $sg_0  \in T$ and $\varphi(sg_0)\neq sg_0$, then
\begin{align*}
\varphi(sg_0)&=\varphi(1\cdot sg_0)=(sg_0)^{-1}=s^{-1}g_0^{-1}\\
			&=\varphi(g_0\cdot s)\in\{g_0s,g_0s^{-1}\}.
\end{align*}
This implies that $g_0^2=1$, a contradiction. 
Hence for every $s\in S$ we have $\varphi(sg_0)=sg_0$  and also $\varphi(s)=s$ by taking $t=sg_0^{-1}$.
By Lemma \ref{Lemma:Trivial}, if $\varphi(g_0)=g_0$ then $\varphi=\id$.
Finally, if $\varphi(g_0)\neq g_0$, then $\eta\circ\varphi$ is the identity, that is $\varphi=\eta$.

For the third assertion, since $S$ generates $Q_8\times B$, it has two elements of order $4$, say $s$ and $t$, that do not commute and hence generate a subgroup isomorphic to $Q_8$. Moreover, the product map $\langle s,t\rangle \times B \to G$ is easily seen to be  an isomorphism. So in other words, we can assume that $s$ and $t$ generate the first factor in the decomposition $G=Q_8 \times B$.

We already know by Lemma~\ref{Lemma:Boolean} that $\xi_G=\xi_{Q_8}$ and that $\Xi_G=\Xi_{Q_8}\times B$. We therefore have to prove that the only element $\varphi \in \xi_{S^{\leq 3}}$ that is the identity on $Q_8$ is the identity on $G$. To that purpose, we apply Lemma~\ref{Lemma:Trivial} for $S=S$, $T=S^{\leq 3}$ and $S_0 = \{s,t,st\}$. Let $\varphi \in \xi_{S^{\leq 3}}$ be such that $\varphi_{|_{S_0}}=\id$. Let $x \in S$ be arbitrary, write $x=x_1b$ with $x_1 \in Q_8$ and $b \in B$. Observe that $x_1$ either belongs to $\{s,t,st\}^{\pm}$, or is of order $2$. In the first case, using that $\varphi(x_1)= x_1$ by assumption and that $b \in S^{\leq 3}$ has order $2$, we have $\varphi(x) = \varphi(x_1) b=x$. In the second case $x \in S$ has order $2$ and therefore $\varphi(x)=x$. Similarly, if $z \in S_0$, we have $\varphi(xz) = \varphi(x_1zb)= x_1zb=xz$. We have proved that for every $\varphi \in \xi_{S^{\leq 3}}$, $\varphi_{|_{S_0}}=\id$ implies $\varphi_{|_{S \cup SS_0}}=\id$. Lemma~\ref{Lemma:Trivial} allows us to conclude.

For the fourth assertion, let $A<G$ be the index $2$ abelian group as in the definition of generalized dicyclic groups. We assume that $S$ is symmetric. Since $S$ is generating, there is an element $x \in S \setminus A$. It is necessarily of order $4$ and $G = \langle A,x\rangle$.

Let us define two symmetric subsets of $A$: $S_0 = A \cap S$ and $S_1 = (x^{-1}S \cup xS) \cap A$. The assumption that $S$ generates $G$ implies that $S_0 \cup S_1$ generates $A$. Moreover, $T\coloneqq \{x,x^{-1}\}\cup S_0 \cup S_1$ generates $G$ and is contained in $S^{\leq2}$.
 
Consider $\varphi \in \xi_{S^{\leq 2}}$. Let $a$ be an element of $T$ of order $2$. Then $\varphi(a)\in\{a,a^{-1}\}=\{a\}$ and for every $g\in G$ we also have $\varphi(ag)=\varphi(ga)=\varphi(g)a=a\varphi(g)$.
 
Consider $\varphi \in \xi_{S^{\leq 3}}$ such that $\varphi(x)=x$.
Observe that for every $a,b \in S_0 \cup S_1$,
  \[ a \in S^{\leq 3}\textrm{ and }xa \in S^{\leq 3} \textrm{ and } \left(a^{-1}b \in S^{\leq 3}\textrm{ or } a^{-1}b x^2 \in S^{\leq 3}\right).\]
  The only case when $a^{-1}b$ might not belong to $S^{\leq 3}$ is when $a=xs$ and $b = x^{-1} t$ (or the converse), and in that case $a^{-1} b x^2 =s^{-1}t \in S^{\leq 2}$. Moreover, in that case, using that $x^2$ has order $2$ and belongs to $S^{\leq 3}$, we have
  \[ \varphi(g a^{-1} b) = \varphi(g a^{-1}b x^2 x^2) = \varphi( g a^{-1} b x^2) x^2 \in \varphi(g) \{a^{-1} bx^2, x^{-2}b^{-1} a\} x^2,\]
  or simply
  \[ \varphi(ga^{-1}b) \in \varphi(g) \{a^{-1}b,b^{-1}a\}.\]
What we retain from this discussion if that, for every $g \in G$ and every $a,b \in S_0 \cup S_1$,
\begin{equation}\label{eq:S3}  \varphi(gh) \in \varphi(g)\{h,h^{-1}\}\textrm{ when }h=a,xa\textrm{ or }a^{-1}b.
\end{equation}

Let $a \in S_0 \cup S_1$. We shall prove, by a case-by-case analysis, that $\varphi(a)=a$ and $\varphi(ax)=ax$.

  If $a^2=1$, since $\xi_{S^{\leq 3}}$ is a subgroup of $\xi_{S^{\leq 2}}$, we already know that $\varphi(a)=a$ and $\varphi(ax)=a\varphi(x)=ax$.

  Consider now the case when $a^2 \notin \{1,x^2\}$. Assume first, for a contradiction, that $\varphi(a) \neq a$. Then using \eqref{eq:S3} we obtain
\[a^{-1} = \varphi(a) = \varphi(x a^{-1} x^{-1}) = x (a^{-1} x^{-1})^{-1} = x^2 a^{-1},\]
which contradicts the assumption that $a^2 \neq x^2$. Assume now, for a contradiction, that $\varphi(ax) \neq ax$. Then using \eqref{eq:S3} we obtain
\[ x^{-1} a^{-1} = \varphi(ax)=\varphi(x a^{-1}) = xa,\]
which again contradicts the assumption that $a^2 \neq x^{2}$. So we have $\varphi(a) = a$ and $\varphi(ax)=ax$ in that case too.

Before we consider the last case, we observe that the second case necessarily happens by Lemma~\ref{lem:gendicyclic_not_Q8}. In particular, there is an element $a_0 \in S_0 \cup S_1$ such that $a_0^2 \notin \{1,x^2\}$. We just proved that $\varphi(a_0)= a_0$ and $\varphi(a_0x)=a_0 x$.

Finally, consider now the case when $a^2=x^2$. Assume for a contradiction that $\varphi(a) \neq a$. Then using \eqref{eq:S3} we obtain
\[ a^{-1} = \varphi(a) = \varphi(a_0 a_0^{-1}a) = a_0 (a_0^{-1}a)^{-1} = a_0^2 a^{-1},\]
a contradiction. So $\varphi(a) = a$. Assume now that $\varphi(ax) \neq ax$.  Then using \eqref{eq:S3} we obtain
\[ x^{-1}a^{-1} = \varphi(ax) = \varphi(a_0 x a^{-1} a_0) = a_0 x (a^{-1}a_0)^{-1} = x a_0^{-2}  a,\]
which implies $a_0^2 = a^2 x^2 = 1$, a contradiction.

So we have indeed proven that, for every $\varphi \in \xi_{S^{\leq 3}}$, $\varphi(x)=x$ implies that $\varphi(a)=a$ and $\varphi(ax)=ax$ for every $a \in S_0 \cup S_1$. By Lemma~\ref{Lemma:Trivial}, we can deduce from this that $\varphi(x)=x$ implies $\varphi = \id$. This concludes the proof of the fourth assertion.

Finally, the last assertion is a direct consequence of \cite[Theorem 7]{LdlS2020}.
\end{proof} 
We conclude this note by a discussion on the optimality of the bounds of Theorem~\ref{Thm:Quantitative}.
Namely, we will show that $S^{\leq2}$ cannot be replaced by $S$ in the second assertion of Theorem~\ref{Thm:Quantitative}, while $S^{\leq 3}$ cannot be replaced by $S^{\leq2}$ in the third, fourth and the fifth assertions.

\begin{example}\label{Example:Product}
For $i\in\{1,2\}$, let  $G_i=\langle S_i\rangle$ be a group with a given generating set not containing the identity and let $G\coloneqq G_1\times G_2$ be their product with generating set $S\coloneqq(S_1\times\{1\})\cup(\{1\}\times S_2)$.
Then for any $\varphi_i\in\xi_{S_i}$, the map $(\varphi_1\times\varphi_2)(g_1,g_2)\coloneqq(\varphi_1(g_1),\varphi_2(g_2))$ is in $\xi_S$.
In particular, if $G_1$ and $G_2$ are two abelian groups, then $G$ is abelian and $\xi_S$ contains $\{\id,\eta,\eta_1\times\id,\id\times\eta_2\}$ where $\eta_i$ is the inverse map on $G_i$.
If moreover both $G_1$ and $G_2$ have an element of order at least $3$, then the elements of $\{\id,\eta,\eta_1\times\id,\id\times\eta_2\}$ are pairwise distinct and $\abs{\xi_S}\geq 4$, which gives us an infinite family of groups showing the optimality of the bound $S^{\leq 2}$ in the second assertion of Theorem~\ref{Thm:Quantitative}.
\end{example}

\begin{example}\label{Ex:2}
Write $\Z/2\Z=\{0,1\}$ in additive notation and let $G_1=Q_8 \times \Z/2\Z$ and $S_1 = \{(\pm i,1),(\pm j,1),(\pm ij,1)\}$ and let $\varphi_0,\varphi_1$ be two disctinct elements of $\xi_{Q_8}$; for example $\varphi_0(s)=s$ and $\varphi_1(s) = s^{-1}$. By Theorem~\ref{Thm:Main2}, the map
\[ \varphi(x,\varepsilon) = (\varphi_\varepsilon(x),\varepsilon)\]
does not belong to $\xi_{G_1}$, but we claim that that it belongs to $\xi_{S_1^{\leq 2}}$. Indeed, we have $S_1^{\leq 2} = S \cup Q_8 \times \{0\}$. The fact that $\varphi(xy) \in \{\varphi(x) y, \varphi(x)y^{-1}\}$ when $y \in Q_8 \times\{0\}$ is the assumption that $\varphi_0$ and $\varphi_1$ belong to $\xi_{Q_8}$. When $y=(t,1) \in S_1$, and $x=(s,\varepsilon)$, we have $t^{-1} = -t$ and so
\begin{align*}\varphi(xy) = (\varphi_{1+\varepsilon}(st),1+\varepsilon) &\in \{(st,1+\varepsilon),(-st,1+\varepsilon)\}\\
  & = \{(\varphi_\varepsilon(s) t,1+\varepsilon),(\varphi_\varepsilon(s) t^{-1},1+\varepsilon) \} \\
  &=\{ \varphi(x)y,\varphi(x)y^{-1}\}. 
\end{align*}
This proves the claim and illustrates the optimality of the bound the bound $S^{\leq 3}$ in the third assertion of Theorem~\ref{Thm:Quantitative}.

  Observe that this single example can be turned into an infinite family, by taking $G_n = Q_8 \times (\Z/2\Z)^n = G_1 \times  (\Z/2\Z)^{n-1}$, $S_n=S_1 \cup \{(1,0)\} \times (\Z/2\Z)^{n-1}$. In that case, the map $\varphi_n(x,\varepsilon,z) = (\varphi(x,\varepsilon),z)$ belongs to $\xi_{S_n^{\leq 2}}$ but not to $\xi_{G_n}$.
In fact, the same argument works for $Q_8\times B$ for any non-trivial boolean group~$B$.
\end{example}
\begin{example}\label{Ex:3}
The following example is taken from \cite{LdlS2020}, where detailed proofs of its properties can be found at the end of Section 3.
For every $n\geq 2$ let
\[
	H_n = \presentation{s_1,\dots,s_n}{\forall i\neq j: s_is_js_i^{-1}=s_j^{-1}}
\]
with generating set $S_n = \{s_1,\dots,s_n\}$ and let $\varepsilon=s_1^2$.
The group $H_n$ has order $2^{n+1}$ and hence the $H_i$ are pairwise distinct.

While $H_2=Q_8$ and $H_3$ are generalized dicyclic (for $H_3$, the index $2$ abelian subgroup is $\langle s_1s_2,s_3,\varepsilon\rangle$), the group $H_{n}$ is never generalized dicyclic for $n\geq 4$.
Nevertheless, for each of these groups, the map $\eta\colon g\to g^{-1}$ is in $\xi_{S_n^{\leq 2}}$.
We hence have an infinite family of groups showing that it is not possible to replace $S^{\leq 3}$ by $S^{\leq 2}$ in the fifth assertion of Theorem~\ref{Thm:Quantitative}.

Finally, the group $H_3$ is not of the form $Q_8\times B$ (as for example it has too many elements of order $4$) and neither is $K_n\coloneqq H_3\times(\Z/2\Z)^n$.
Take $T_n\coloneqq (S_3\times \{1\})\cup(\{1\}\times (\Z/2\Z)^n)$ for a generating set.
We have $T_n^{\leq 2}=(S_3^{\leq2}\times \{1\})\cup(\{1\}\times (\Z/2\Z)^n)\cup(S_3\times (\Z/2\Z)^n)$.
We claim that $\eta$ belongs to $\xi_{T_n}$.
This results follows from an adaptation of Lemma~\ref{Lemma:Boolean}, but here is a detailed proof of it.
Let $(g,h)$ be in $K_n$ and $(x,y)$ be in $T_n$.
Then $\eta((g,h)(x,y))=(x^{-1}g^{-1},hy)$ while $\eta(g,h)\cdot(x,y)^{\pm1}=(g^{-1}x^{\pm1},hy)$.
The desired equality directly follows from the fact that $\eta$ belongs to $\xi_{S_3^{\leq 2},H_3}$.
We hence have exhibited an infinite family of groups showing the optimality of the bound $S^{\leq 3}$ in the fourth assertion of Theorem~\ref{Thm:Quantitative}.

Observe that for $\alpha$ an arbitrary infinite cardinal, one can adapt the above examples to obtain groups $H_\alpha$ and $K_\alpha$ of cardinality $\alpha$ with the desired properties.
\end{example}
%
%
%
%
%
%
%
%
%
%
%
%
%%%%%%%%%%%%%%%%%%%%% Bibliography %%%%%%%%%%%%%%%%%%%%
\providecommand{\noopsort}[1]{} \def\cprime{$'$}

\end{document}